\documentclass[letterpaper, 10 pt, conference]{ieeeconf}  

\IEEEoverridecommandlockouts                              
\usepackage{graphics} 
\usepackage{amsmath} 
\usepackage{amssymb}  

\usepackage[update,prepend]{epstopdf}
\usepackage{caption}
\usepackage{subcaption}
\usepackage{multirow}

\newcommand{\1}{{\rm 1\hspace*{-0.4ex}%
\rule{0.1ex}{1.52ex}\hspace*{0.2ex}}}

\usepackage{amsthm}

\theoremstyle{plain}
\newtheorem{theorem}{Theorem}

\newtheorem{corollary}{Corollary}
\newtheorem{proposition}{Proposition}

\theoremstyle{definition}
\newtheorem{optimization}{Optimization}

\theoremstyle{remark}
\newtheorem{remark}{Remark}

\title{\LARGE \bf
Incentive Design for Efficient Building Quality of Service
}

\author{Anil Aswani$^{1}$ and Claire Tomlin$^{1}$
\thanks{*This material is based upon work supported by the National Science Foundation under Grant CNS-0931843 (CPS-ActionWebs). The views and conclusions contained in this document are those of the authors and should not be interpreted as representing the official policies, either expressed or implied, of the National Science Foundation.}
\thanks{$^{1}$A. Aswani and C. Tomlin are with the Department of Electrical Engineering and Computer Sciences, University of California, Berkeley CA 94720, USA {\tt\small \{aaswani,tomlin\} at eecs.berkeley.edu}}%
}

\begin{document}

\maketitle
\thispagestyle{empty}
\pagestyle{empty}

\begin{abstract}
Buildings are a large consumer of energy, and reducing their energy usage may provide financial and societal benefits.  One challenge in achieving efficient building operation is the fact that few financial motivations exist for encouraging low energy configuration and operation of buildings.  As a result, incentive schemes for managers of large buildings are being proposed for the purpose of saving energy.  This paper focuses on incentive design for the configuration and operation of building-wide heating, ventilation, and air-conditioning  (HVAC) systems, because these systems constitute the largest portion of energy usage in most buildings.  We begin with an empirical model of a building-wide HVAC system, which describes the tradeoffs between energy consumption, quality of service (as defined by occupant satisfaction), and the amount of work required for maintenance and configuration.  The model has significant non-convexities, and so we derive some results regarding qualitative properties of non-convex optimization problems with certain partial-ordering features.  These results are used to show that ``baselining'' incentive schemes suffer from moral hazard problems, and they also encourage energy reductions at the expense of also decreasing occupant satisfaction.  We propose an alternative incentive scheme that has the interpretation of a performance-based bonus.  A theoretical analysis shows that this encourages energy and monetary savings and modest gains in occupant satisfaction and quality of service, which is confirmed by our numerical simulations.
\end{abstract}

\section{Introduction}

Some large buildings have dedicated, full-time staff (called building managers) who are tasked with configuration and maintenance of the building.  These building managers are evaluated on their ability to properly regulate security, safety, and occupant satisfaction of the building.  On the other hand, their roles have traditionally not included achieving sustainable building operation, especially since this objective can be counter to occupant satisfaction (though safety and security do not substantially affect building energy consumption).  

Though building electricity use currently constitutes a small percentage of total operating expenses of corporations, achieving energy savings can bestow significant financial and societal benefits.  For instance, building electricity usage constitutes seventy percent of total usage in the United States \cite{buildings2009}.  As a result, monetary incentive schemes for improving the sustainability of organizations utilizing large buildings are slowly being introduced in some domains.  

This paper analyzes one proposed incentive scheme for encouraging energy-efficient configuration by building managers and then uses the insights gained to design another incentive scheme with better properties.  We focus on incentives for building managers specifically applied to the energy consumption of heating, ventilation, and air-conditioning (HVAC) systems, and we have two reasons for this.  First, HVAC constitutes forty percent of total energy used in a building \cite{buildings2009}.  Second, HVAC explicitly admits a quality of service signal, in addition to energy usage, that can be measured and considered in the incentive design \cite{callaway2011}.


Incentives for the electricity grid are similar to the single building scenario.  Such approaches can be classified \cite{bushnell2009} into either dynamic pricing or demand response approaches.  The distinction is that demand response approaches establish a baseline of electricity usage and then provide incentives relative to this baseline (e.g., \cite{lawrence2002,ontario2007,lbnl2011}).  In contrast, dynamic pricing involves treating consumers and producers in a symmetric manner by pricing electricity to accurately reflect generation costs and uncertainty along with demand elasticity (e.g., \cite{schweppe1980,kirschen2003,rooz2010,mathieu2011,libin2011,huang2011,saad2012,lavaei2012}).

\subsection{Demand Response of Electricity}

In this approach, a baseline of electricity consumption for a consumer is established; next, consumers are rewarded (or penalized) for reducing (increasing) their energy usage from the baseline amount.  It is well known \cite{bushnell2009} that this approach suffers two major problems.  \textit{Adverse selection} occurs whenever a consumer is given a reward for a reduction that they would have made without the incentives.  This can occur, for instance, when a child leaves a home to go to college.  \textit{Moral hazard} situations arise because consumers are encouraged by the incentive (whenever they know about the policy in advance) to artificially inflate electricity usage during the period in which the baseline amount is established; the reason is that the rewards achieved are relative to the baseline period, and so a higher baseline will mean a higher reward for a constant level of post-baseline energy savings. 

\subsection{Comparison of Static and Dynamic Models}

This paper begins by performing a Monte Carlo analysis in order to convert a hybrid system dynamical model of a building-wide HVAC system and its energy characteristics into a static model that describes the operational characteristics of the HVAC system. Our static model has only two states, which is fewer than the dynamical model; as a result, it provides greater intuition about the problem.  The static model is easier to analyze with game-theory because temporal effects typically complicate analysis.  Moreover, the static model more closely reflects the fact that the agents utilize fixed control strategies (i.e., the configuration is kept constant) when commissioning a building.   

This is not to say that the dynamic model is inferior.  Such models have been used to experimentally implement energy-efficient HVAC controllers \cite{aswani2012_brites}.  Moreover, incentive design to encourage dynamic building configuration strategies will lead to additional energy savings not possible with the techniques discussed here.  This is an open area.


\subsection{Overview}
The static model has significant non-convexities, and so we derive some results on qualitative properties of non-convex optimization problems with certain partial-ordering properties.  These results are used to analyze models of the utilities of building managers and building owners.  Next, we analyze one scheme that has been proposed for incentivizing building managers to seek reductions in building energy usage; this incentive is shown to suffer from moral hazard problems.  Furthermore, this scheme provides reductions in energy consumption at the expense of occupant satisfaction.  In response, we propose an alternative incentive that does not suffer from moral hazard, and this scheme encourages energy and monetary savings along with improvements to  occupant satisfaction; this is verified with numerical simulations.

\subsection{Notation}

A subscript on a variable denotes a different time step.  For instance, let $E$ be the energy usage of a system.  Then, $E_1$ and $E_2$ refer to energy usage in time periods one and two.  On the other hand, superscript denotes different values of variables.  For instance, a superscript star $^*$ denotes the maximizer of specified optimization problems.  As another example, superscript $\text{min},\text{max}$ denote the extent of variables.  Also, we use the convention that the property of hemicontinuity also includes compactness.

\section{Modeling Building HVAC Operation}

\begin{figure*}
\begin{subfigure}[b]{0.33\textwidth}
\centering
\includegraphics{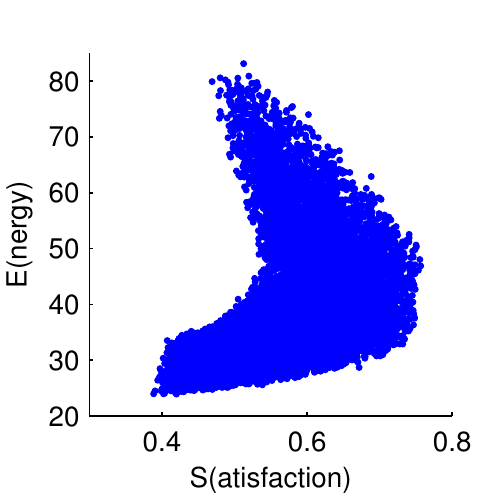}
\caption{Scatter Plot}
\label{fig:scatter}
 \end{subfigure}
\begin{subfigure}[b]{0.33\textwidth}
\centering
\includegraphics{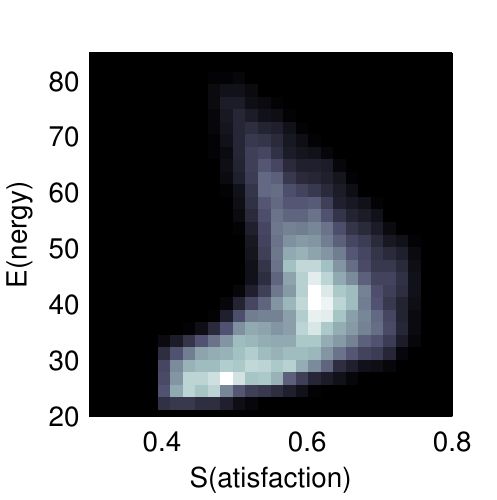}
\caption{Density Plot}
\label{fig:density}
 \end{subfigure}
\begin{subfigure}[b]{0.33\textwidth}
\centering
\includegraphics{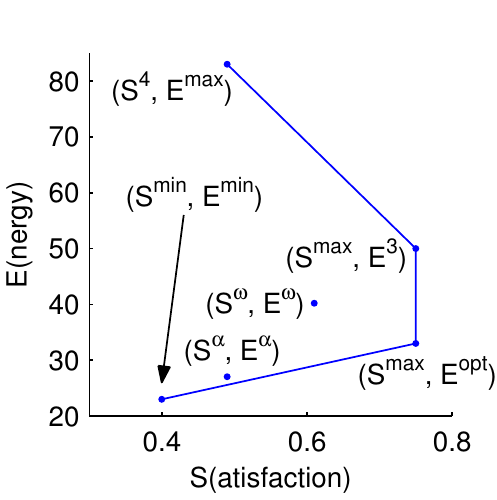}
\caption{Key Points}
\label{fig:main_points}
 \end{subfigure}

\caption{\label{fig:monte_model}A Monte Carlo analysis of thermal and energy models for BRITE-S was used to generate a scatter plot of different operating points of the BRITE-S system.  Its corresponding density plot is shown, and more white values indicate greater densities.  The key points of the static, operational model as determined by the Monte Carlo analysis are also shown.}
\end{figure*}

We begin this section by summarizing a model of thermal dynamics and energy consumption of a building-wide HVAC testbed named BRITE-S \cite{aswani2012_brites} with variable air volume (VAV) equipment, which was used to experimentally implement an energy saving controller.  Next, we conduct a Monte Carlo analysis using this model to construct a static model that is suitable for a game-theoretic analysis.

\subsection{Thermal Dynamics of BRITE-S}

Let $T,F,R,Q \in \mathbb{R}^{n}$ be vectors of temperatures, air flow rates, air reheat amount, and amount of heating load at $n$ different locations in the building.  A hybrid system model is 
\begin{equation}
\label{eqn:dynamics}
T_{k+1} = A\langle m_k \rangle \cdot T_k + B\langle m_k \rangle \cdot F_k+ C\langle m_k \rangle \cdot R_k + Q_k,
\end{equation}
where $A\langle m \rangle \in \mathbb{R}^{n \times n}, B\langle m \rangle \in \mathbb{R}^{n \times n}, C\langle m \rangle \in \mathbb{R}^{n \times n}$ are sets of matrices indexed by $m \in \{1,2,3\}$; this denotes that the hybrid system model has three discrete modes.  

Furthermore, the air flow rates $F$ and reheat amount $R$ are functions of the current temperature $T$ and the desired temperature $T_d \in \mathbb{R}^n$.  This is approximately modeled as a linear function of $T - T_d$ with saturation: More details can be found in \cite{aswani2012_brites}.  Furthermore, each point in the building has constraints on the minimum (maximum) amount of airflow $F \geq F^{\text{min}} \in \mathbb{R}^n$ ($F \leq F^{\text{max}} \in \mathbb{R}^n$).

\subsection{Energy Model of BRITE-S}

Building HVAC systems have several pieces of equipment that contribute to the overall energy usage.  Within BRITE-S, most energy consumption is due to fans that distribute air, chillers that cool the air, and the reheating of air at different points in the building.  Their respective sum is
\begin{equation}
E_k = a(\1^T F_k)^3 + b(t_s\langle m\rangle - o_k)(\1^T F_k) + c(\1^T R_k),
\end{equation}
where $\1$ is a vector of all ones, the superscript $T$ denotes a matrix transpose, $a,b,c$ are positive constants, $t_s\langle m \rangle$ is the temperature to which the air is centrally cooled in the building ($m \in \{1,2,3\}$ is the discrete mode of the thermal dynamics), and $o_k$ is the outside air temperature.

\subsection{Satisfaction Model of BRITE-S}

Let $(x)_+$ be the thresholding function, which is defined so that $(x)_+ = 0$ if $x < 0$ and $(x)_+ = x$ otherwise; we used the following quantification
\begin{equation}
S = 1 - \frac{1}{nk}\textstyle \sum_k \1^T (|T_k - T_d| - B)_+,
\end{equation}
where $B \in \mathbb{R}^n$ is an amount of temperature deviation that does not cause discomfort.  The intuition of this quantity is that satisfaction decreases whenever the temperature deviation exceeds the desired temperatures $T_d$ by more than $B$, and the amount of decrease in satisfaction is proportional to the magnitude and duration of this deviation.

\subsection{Monte Carlo Analysis of BRITE-S}

The variables $F^{\text{min}},F^{\text{max}},m$ correspond to different configurations of the HVAC equipment, and the variables $o_k,Q_k$ relate to variations in weather and occupancy conditions.  We conducted a Monte Carlo analysis in which these variables were randomly varied according to a uniform distribution whose extents are physically realistic values.   For different instantiations of these variables, the energy consumption $E = \sum_k E_k$ and occupant satisfaction $S$ over one day were computed.  Varying the distribution used in this analysis did not significantly change the qualitative features of the results; this is significant, because our theoretical results and analysis depend only upon these qualitative features.

\subsection{Static, Operational Model of BRITE-S}

The results of the Monte Carlo analysis are summarized in Fig. \ref{fig:monte_model}.  A scatter plot (i.e., Fig. \ref{fig:scatter}) indicates the possible range of the HVAC equipment operation in the space of energy usage $E$ versus satisfaction $S$.  We denote this region of feasible operating points, which can be bounded by a nonconvex polygon $\mathcal{O}$.  The density of the points in the scatter plot is shown in Fig. \ref{fig:density}.  

We model the amount of work $W(S,E)$ required by the building manager to configure the building to operate at the point $(S,E)$ as being inversely proportional to the density of the scatter plot.  This is an approximate model that we pick because it (i) describes the intuition that achieving $(S,E)$ is more difficult if fewer configurations keep the building operating at that point, and (ii) admits a tractable analysis.  For simplicity, we normalize this function so that $W(S,E) \in [0,1]$ and $W(S,E) \geq 0$ for all $(S,E) \in \mathcal{O}$.  Also, we assume that $W(S,E)$ is continuous.

The key points of this model (besides the set of feasible operating points $\mathcal{O}$) are shown in Fig. \ref{fig:main_points}.  Our model features
\begin{enumerate}

\item Two Isolated Minima: $(S^\alpha,E^\alpha),(S^\omega,E^\omega) \in \text{int}(\mathcal{O})$, which obey the relationship that $W(S^\alpha,E^\alpha) = W(S^\omega,E^\omega) < W(S,E)$ for all $(S,E) \in \mathcal{O} \setminus \{(S^\alpha,E^\alpha),(S^\omega,E^\omega)\}$;

\item Minima Ordering: $S^\alpha < S^\omega$ and $E^\alpha < E^\omega$;

\item Linear Bounds: $(S^{\text{min}},E^{\text{min}}),(S^{\text{max}},E^{\text{opt}}),(S^{\text{max}},E^3)$, and $(S^4,E^\text{max})$ form the right boundary of $\mathcal{O}$.

\end{enumerate}

This model has some associated intuition.  The smaller minimum $(S^\alpha,E^\alpha)$ corresponds to configurations in which the HVAC system is turned off --- though energy is easily saved, this comes at the cost of reduced occupant satisfaction.  The larger minimum $(S^\omega,E^\omega)$ corresponds to configurations in which the HVAC system is minimally configured; the satisfaction and energy usage are of moderate levels.  

The upper portion of $\mathcal{O}$ is related to the fact that the HVAC system can be configured to purposely waste energy, but this does cause comfort to be reduced; an example of this is increasing the amount of cooling in the building to unreasonable levels.  Lastly, there is a maximum amount of satisfaction that can be achieved and a variety of configurations that can achieve this.  This occurs in building-wide HVAC because some configurations simultaneously cool and heat air, which wastes energy without affecting comfort.


\section{Some Generic Non-Convex Optimization}

We provide results about qualitative properties of optimization problems with certain generic properties.  Essentially, if certain terms in the objective have order-preserving properties (e.g., strictly increasing or strictly decreasing functions), then certain components of the minimizing argument also obey order-preserving properties.  For our purposes, we assume that all variables involved in the optimization are scalars.  These results naturally extend to functions of vector-valued functions that are order-preserving under appropriately defined partial orders (e.g., convex cone partial orders).

\begin{optimization}
\label{opt:gen1}
\begin{equation}
\label{eqn:generic}
\max_{x,y}\{f(x) + \lambda g(x,y) : (x,y) \in \mathcal{S}\},
\end{equation}
where $x,y\in\mathbb{R}$, $\mathcal{S}$ is a closed and bounded set, and both $f(x)$ and $g(x,y)$ are continuous functions. 
\end{optimization}

\begin{remark}
\label{remark:one}
Note that these assumptions allow us to apply the Berge maximum theorem \cite{berge1963}, which gives that
\begin{enumerate}
\item the maximum (and minimum) of the objective in (\ref{eqn:generic}) restricted to the set $\mathcal{S}$ is attained;
\item the set of maximizers $\mathcal{M}(\lambda) = \arg\ (\ref{eqn:generic})$ is upper hemicontinuous.
\end{enumerate}
\end{remark}

\begin{remark}
\label{remark:two}
The projection of the set of maximizers onto its $x$-component, denoted $x^*(\lambda) = \text{Proj}_x(\mathcal{M}(\lambda))$, is also upper hemicontinuous because of the closed map lemma.  Applying the Berge maximum theorem again shows that the maximum (and minimum) of the set $x^*(\lambda)$ is attained.  We define
\begin{align}
\overline{x}^*(\lambda) &= \max \{x : x \in x^*(\lambda)\} \\
\underline{x}^*(\lambda) &= \min \{x : x \in x^*(\lambda)\}.
\end{align}
\end{remark}

\begin{theorem}
\label{theorem:gen1}
Consider the generic problem Optimization \ref{opt:gen1}.  If $f(x)$ is strictly increasing in $x$ and $\lambda \geq 0$, then $\underline{x}^*(\lambda^1)  \geq \overline{x}^*(\lambda^2)$ for $0 \leq \lambda^1 < \lambda^2$.
\end{theorem}

\begin{proof}
Suppose this theorem were not true.  Then there exists $0\leq \lambda^1 < \lambda^2$ such that $\underline{x}^*(\lambda^1) < \overline{x}^*(\lambda^2)$.  We show that this results in a contradiction.

The first case occurs when $\lambda^1 = 0$.  Since $f(\cdot)$ is strictly increasing, the corresponding maximizer is uniquely given by $x^*(0) = \max \{x : (x,y) \in \mathcal{S}\}$.  As a result, it must be that $\overline{x}^*(\lambda^2) \leq x^*(0)$ for any $\lambda^2 > 0$.  This is a contradiction.

The second case occurs when $\lambda^1 > 0$.  Because there exists some $\nu^*$ such that $(\underline{x}^*(\lambda^1), \nu^*)$ is a maximizer for $\lambda^1$, this means that
\begin{equation}
\label{eqn:m1}
f(\underline{x}^*(\lambda^1)) + \lambda^1 g(\underline{x}^*(\lambda^1),\nu^*) \geq f(x) + \lambda^1 g(x,y),
\end{equation}
for all $(x,y) \in \mathcal{S}$.  For any $(x,y) \in \mathcal{S} : x > \underline{x}^*(\lambda^1)$, the hypothesis on the strictly increasing characteristic of $f(\cdot)$ and the inequality (\ref{eqn:m1}) imply that
\begin{equation}
\label{eqn:m3}
g(\underline{x}^*(\lambda^1),\nu^*) > g(x,y).
\end{equation}

Next, adding $(\lambda^2-\lambda^1)g(x,y)$ to both sides of the inequality (\ref{eqn:m1}) and then simplifying gives
\begin{align}
\label{eqn:m2}
&f(\underline{x}^*(\lambda^1)) + \lambda^1 g(\underline{x}^*(\lambda^1),\nu^*) + (\lambda^2-\lambda^1)g(x,y) \\
& \quad \geq f(x) + \lambda^2 g(x,y).
\end{align}
Because this holds for all $(x,y) \in \mathcal{S}$, it also holds for any $(x,y) \in \mathcal{S} : x > \underline{x}^*(\lambda^1)$.

Now consider maximizing (\ref{eqn:m2})  over $(x,y) \in \mathcal{S} : x > \underline{x}^*(\lambda^1)$.  Because $\lambda^2 > \lambda^1$, this is equivalent to maximizing $g(x,y)$ over $(x,y) \in \mathcal{S} : x > \underline{x}^*(\lambda^1)$.  However, the inequality (\ref{eqn:m3}) implies that the supremum value over this range is $g(\underline{x}^*(\lambda^1),\nu^*)$.  Consequently, $f(\underline{x}^*(\lambda^1)) + \lambda^2 g(\underline{x}^*(\lambda^1),\nu^*) > f(x) + \lambda^2 g(x,y)$ for all $(x,y) \in \mathcal{S} : x > \underline{x}^*(\lambda^1)$.  This is a contradiction when assuming $\underline{x}^*(\lambda^1) < \overline{x}^*(\lambda^2)$ and $\lambda^1 > 0$, because there exists some $\eta^*$ such that $(\overline{x}^*(\lambda^2), \eta^*)$ is a maximizer for $\lambda^2$.
\end{proof}

\begin{corollary}
\label{corollary:gen1}
Consider the generic problem Optimization \ref{opt:gen1}.  If $f(x)$ is strictly decreasing in $x$ and $\lambda \geq 0$, then $\overline{x}^*(\lambda^1) \leq \underline{x}^*(\lambda^2)$ for $0 \leq \lambda^1 < \lambda^2$.
\end{corollary}

\begin{proof}
If we define $\tilde{x} = -x$, $\tilde{f}(x) = f(-x)$, $\tilde{g}(x,y) = g(-x,y)$, and $\tilde{\mathcal{S}} = \{(-x,y) : (x,y) \in \mathcal{S}\}$; then the result follows by applying Theorem \ref{theorem:gen1} to the optimization 
\begin{equation*}
\max_{\tilde{x},y}\{\tilde{f}(\tilde{x}) + \lambda \tilde{g}(\tilde{x},y) : (\tilde{x},y) \in \tilde{\mathcal{S}}\}. \qedhere
\end{equation*}
\end{proof}

\begin{remark}
Note that the results in Theorem \ref{theorem:gen1} and Corollary \ref{corollary:gen1} are ``tight'' in the sense that the maximizers can be multi-valued for fixed values of $\lambda$, even when the objective is linear.  An example of this is seen in Sect. \ref{section:owner}.
\end{remark}

\begin{optimization}
\label{opt:gen2}
\begin{equation}
\label{eqn:generic2}
\max_{x,y}\{\lambda f(x) + g(x,y) : (x,y) \in \mathcal{S}\},
\end{equation}
where $x,y\in\mathbb{R}$, $\mathcal{S}$ is a closed and bounded set, and both $f(x)$ and $g(x,y)$ are continuous functions. 
\end{optimization}

\begin{remark}
Under these assumptions, the facts in Remarks \ref{remark:one} and \ref{remark:two} also hold for Optimization \ref{opt:gen2}.
\end{remark}

\begin{figure*}
\begin{subfigure}[b]{0.33\textwidth}
\centering
\includegraphics{main_points.pdf}
\caption{Key Points}
\label{fig:main_points2}
 \end{subfigure}
\begin{subfigure}[b]{0.33\textwidth}
\centering
\includegraphics{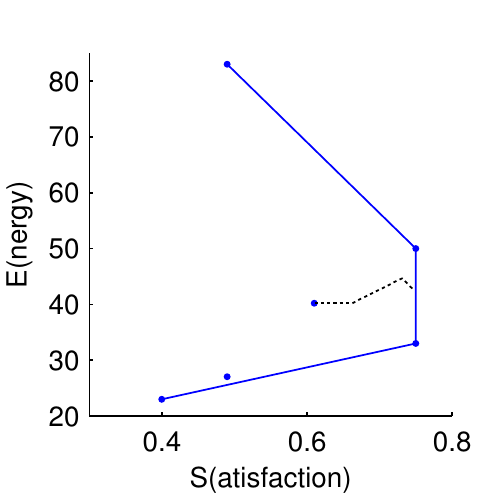}
\caption{Manager Trace}
\label{fig:manager_trace}
 \end{subfigure}
\begin{subfigure}[b]{0.33\textwidth}
\centering
\includegraphics{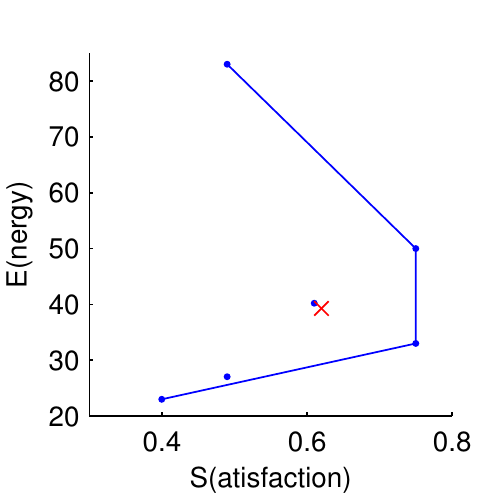}
\caption{Actual Operating Point}
\label{figure:real}
 \end{subfigure}

\caption{\label{fig:monte_model2}The (a) key points of the static, operational model as determined by the Monte Carlo analysis, and (b) a trace (indicted by the dotted line) of operating points for (\ref{eqn:no_incentive}) as $\lambda$ is varied are shown.  Also, the actual operating point of the building as configured in real life is indicated by the cross mark.}
\end{figure*}

\begin{theorem}
\label{theorem:gen2}
Consider the generic problem Optimization \ref{opt:gen2}.  If $f(x)$ is strictly increasing in $x$ and $\lambda \geq 0$, then $\overline{x}^*(\lambda^1) \leq \underline{x}^*(\lambda^2)$ for $0 \leq \lambda^1 < \lambda^2$.
\end{theorem}

\begin{proof}
Suppose $\lambda^1 > 0$.  If we define $\tilde{\lambda}^2 = 1/\lambda^1$ and $\tilde{\lambda}^1 = 1/\lambda^2$; then the result follows by applying Theorem \ref{theorem:gen1} to the optimization 
\begin{equation*}
\max_{x,y}\{f(x) + \tilde{\lambda} g(x,y) : (x,y) \in \mathcal{S}\}.
\end{equation*}

Next suppose that $\lambda^1 = 0$ and this theorem were not true.  Then there exists $0= \lambda^1 < \lambda^2$ such that $\overline{x}^*(\lambda^1) > \underline{x}^*(\lambda^2)$.  We show that this results in a contradiction.

Because there exists some $\nu^*$ such that $(\overline{x}^*(\lambda^1), \nu^*)$ is a maximizer for $\lambda^1$, this means that $g(\overline{x}^*(\lambda^1),\nu^*) \geq g(x,y)$, for all $(x,y) \in \mathcal{S}$.  Thus, it also holds for any $(x,y) \in \mathcal{S} : x < \overline{x}^*(\lambda^1)$ that
\begin{equation}
\label{eqn:m32}
g(\overline{x}^*(\lambda^1),\nu^*) \geq g(x,y).
\end{equation}

Next, adding $(\lambda^2-\lambda^1)f(x)$ to both sides of the inequality (\ref{eqn:m32}) and then simplifying gives
\begin{align}
\label{eqn:m22}
& (\lambda^2-\lambda^1)f(x) + g(\overline{x}^*(\lambda^1),\nu^*) \\
& \quad \geq \lambda^2 f(x) + g(x,y).
\end{align}
Because this holds for all $(x,y) \in \mathcal{S}$, it also holds for any $(x,y) \in \mathcal{S} : x < \overline{x}^*(\lambda^1)$.

Now consider maximizing (\ref{eqn:m22})  over $(x,y) \in \mathcal{S} : x < \overline{x}^*(\lambda^1)$.  Because $\lambda^2 > \lambda^1$, this is equivalent to maximizing $f(x)$ over $(x,y) \in \mathcal{S} : x < \overline{x}^*(\lambda^1)$.  However, the strictly increasing nature of $f(x)$ means that the supremum value over this range is $f(\overline{x}^*(\lambda^1))$.  Consequently, $\lambda^2f(\overline{x}^*(\lambda^1)) + g(\overline{x}^*(\lambda^1),\nu^*) > \lambda^2 f(x) + g(x,y)$ for all $(x,y) \in \mathcal{S} : x < \overline{x}^*(\lambda^1)$.  This is a contradiction when assuming $\overline{x}^*(\lambda^1) > \underline{x}^*(\lambda^2)$ and $\lambda^1 = 0$, because there exists some $\eta^*$ such that $(\underline{x}^*(\lambda^2), \eta^*)$ is a maximizer for $\lambda^2$.
\end{proof}

\begin{corollary}
\label{corollary:gen2}
Consider the generic problem Optimization \ref{opt:gen2}.  If $f(x)$ is strictly decreasing in $x$ and $\lambda \geq 0$, then $\underline{x}^*(\lambda^1) \geq \overline{x}^*(\lambda^2)$ for $0 \leq \lambda^1 < \lambda^2$.
\end{corollary}

\begin{proof}
If we define $\tilde{x} = -x$, $\tilde{f}(x) = f(-x)$, $\tilde{g}(x,y) = g(-x,y)$, and $\tilde{\mathcal{S}} = \{(-x,y) : (x,y) \in \mathcal{S}\}$; then the result follows by applying Theorem \ref{theorem:gen2} to the optimization 
\begin{equation*}
\max_{\tilde{x},y}\{\lambda \tilde{f}(\tilde{x}) + \tilde{g}(\tilde{x},y) : (\tilde{x},y) \in \tilde{\mathcal{S}}\}. \qedhere
\end{equation*}
\end{proof}


\section{Model of Agents}

Here, we describe the utilities of the building manager and building owner.  Relevant features of these utilities are described, and the moral hazard present in one proposed incentive scheme is shown.

\subsection{Model of Building Manager}

The typical building is operated to maintain maximal occupant satisfaction and comfort, because building managers' job performance is partly evaluated on the basis of occupant complaints.  More specifically, we model the manager as maximizing the following utility
\begin{equation}
\label{eqn:no_incentive}
\max_{S_i,E_i} \{ \textstyle\sum_i S_i - \lambda \cdot W(S_i,E_i) : (S_i,E_i) \in \mathcal{O} \},
\end{equation}
where $\lambda \geq 0 $ is a constant that determines the personal trade-off for the building manager between the amount of work $W$ required to maintain the building at an operating point of $(S_i,E_i)$ and the occupant satisfaction $S_i$.  The subscripts refer to multiple iterations (e.g., months) of this process; without loss of generality, we assume two iterations: $i = 1,2$.

\begin{proposition}
\label{theorem:first}
If $0 \leq \lambda^1 < \lambda^2$, then the corresponding maximizers of (\ref{eqn:no_incentive}) are nonincreasing:
\begin{equation}
\label{eqn:ordering}
S^{\text{max}} \geq \underline{S}^*(\lambda^1) \geq \overline{S}^*(\lambda^2) \geq S^\omega.
\end{equation}
\end{proposition}

\begin{proof}
The upper bound $S^{\text{max}}$ is obvous, and so we first show that the maximizers satisfy $\underline{S}^*(\lambda) \geq S^\omega$ for any $\lambda \geq 0$.  By assumption $W(S,E) \geq W(S^\omega,E^\omega)$ for all $(S,E) \in \mathcal{O}$, which implies that $S^\omega - \lambda\cdot  W(S^\omega,E^\omega) > S - \lambda\cdot W(S,E)$ for all $S < S^\omega$.  Therefore, $\underline{S}^*(\lambda) \geq S^\omega$.  The remainder of the ordering result (\ref{eqn:ordering}) follows from Theorem \ref{theorem:gen1}.
\end{proof}

%

\begin{proposition}
\label{proposition:converge}
The maximizers converge
\begin{align}
\lim_{\lambda \rightarrow \infty} \underline{S}^*(\lambda) &= \lim_{\lambda \rightarrow \infty} \overline{S}^*(\lambda) = S^\omega \\
\lim_{\lambda \rightarrow \infty} \underline{E}^*(\lambda) &= \lim_{\lambda \rightarrow \infty} \overline{E}^*(\lambda) = E^\omega.
\end{align}
\end{proposition}

\begin{proof}
We define an $L^1$ ball of radius $\rho$ about the point $(S^\omega,E^\omega)$ as the set
\begin{equation}\mathcal{B}(\rho) = \{(S,E) : S^\omega \leq S < S^\omega + \rho \textstyle\bigwedge |E - E^\omega| < \rho\}.
\end{equation}
Next, define the quantity
\begin{multline}
\epsilon(\rho) = \min\{W(S,E) - W(S^\omega,E^\omega) : S \geq S^\omega \textstyle\bigwedge \nonumber\\
(S,E) \in \mathcal{O} \setminus \mathcal{B}(\rho)\}.
\end{multline}
Note that the Berge maximum theorem implies that this minimum is attained on the corresponding constraint set.

By construction, whenever $\lambda > (S^{\text{max}} - S^\omega)/\epsilon(\rho)$, it holds that $S^\omega - \lambda\cdot  W(S^\omega,E^\omega) > S - \lambda\cdot W(S,E)$ for all $S \geq S^\omega \textstyle\bigwedge (S,E) \in \mathcal{O} \setminus \mathcal{B}(\rho)$.  Rewriting this --- given any $\rho > 0$, there exists $\Lambda(\rho) =  (S^{\text{max}} - S^\omega)/\epsilon(\rho)$ such that the maximum belongs to the set $\mathcal{B}(\rho)$ for all $\lambda > \Lambda(\rho)$.  This is the definition of convergence, and so the result is proved.
\end{proof}

\begin{proposition}
\label{proposition:zero}
The maximizers of (\ref{eqn:no_incentive}) for $\lambda = 0$ are given by $(S^*(0),E^*(0)) = \{(S^\text{max},E) : E^\text{opt} \leq E \leq E^3\}$.
\end{proposition}

\begin{proof}
This follows by direct maximization of (\ref{eqn:no_incentive}).
\end{proof}

\begin{remark}
The intuition of Propositions \ref{theorem:first}, \ref{proposition:converge}, and \ref{proposition:zero} is that satisfaction $S$ decreases as the building manager puts more emphasis on reducing his own work $W(S,E)$.  At the extreme of zero emphasis on work $\lambda = 0$, the building will be configured to ensure maximum satisfaction $S^* = S^\text{max}$; and at the other extreme of all emphasis on work $\lambda \rightarrow \infty$, the building will be configured to achieve the global minimum of work that has the higher level of satisfaction $(S^\omega,E^\omega)$.
\end{remark}

Based on the intuition of these propositions, we can examine the operating point of the building as actually configured.  This point is marked in Fig. \ref{figure:real} with a cross.  We can estimate the corresponding $\lambda$ value as the one for which $(S^*(\lambda),E^*(\lambda))$ is closest to the actual operating point.  This is useful for conducting further numerical analysis.




\subsection{Model of Building Owner}
\label{section:owner}

\begin{figure*}
\begin{subfigure}[b]{0.33\textwidth}
\centering
\includegraphics{main_points.pdf}
\caption{Key Points}
\label{fig:main_points3}
 \end{subfigure}
\begin{subfigure}[b]{0.33\textwidth}
\centering
\includegraphics{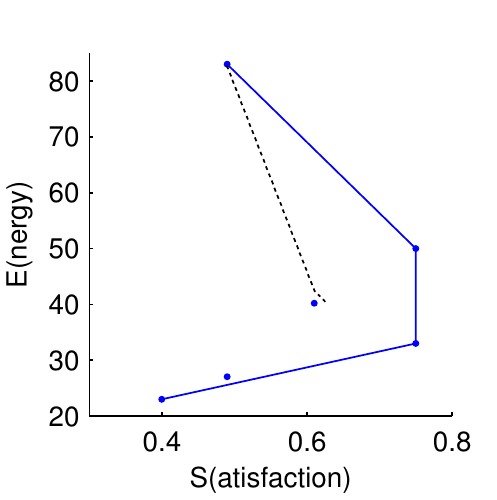}
\caption{Period 1 Trace of (\ref{eqn:bad_inc})}
\label{fig:manager_trace3}
 \end{subfigure}
\begin{subfigure}[b]{0.33\textwidth}
\centering
\includegraphics{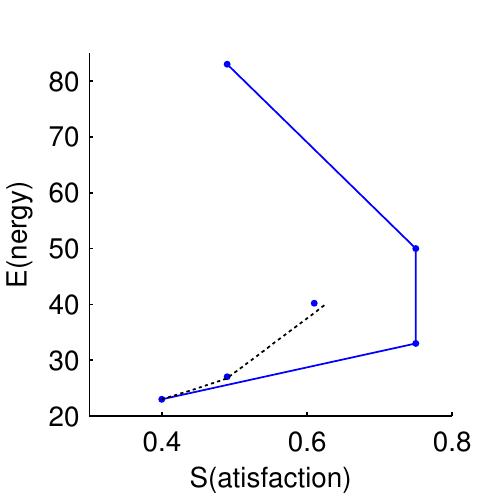}
\caption{Period 2 Trace of (\ref{eqn:bad_inc})}
\label{figure:real3}
 \end{subfigure}

\caption{\label{fig:monte_model3}The key points of the static, operational model as determined by the Monte Carlo analysis are shown.  Furthermore, traces (indicted by the dotted line) of operating points as $\gamma$ is varied are shown for both period 1 and period 2 of operation of the incentive scheme in (\ref{eqn:bad_inc}).}
\end{figure*}

From the perspective of the building owner, the problem with the utility of the building manager (\ref{eqn:no_incentive}) is that it does not take into account the energy consumption of the building.  The owner's preference is that the following utility be maximized
\begin{equation}
\begin{aligned}
\max_{S_i,E_i} \{ \textstyle\sum_i S_i - \mu E_i : (S_i,E_i) \in \mathcal{O} \},
\end{aligned}
\end{equation}
where $\mu > 0$ is a constant that determines the trade-off for the building owner between the amount of energy used and the occupant satisfaction.

\begin{proposition}
The maximizers are given by the following multi-valued function
\begin{multline}
(S^*_i(\mu),E^*_i(\mu)) = \\
\begin{cases} (S^\text{max},E) : E^\text{opt} \leq E \leq E^3 & \text{if } \mu = 0\\
(S^{\text{max}},E^\text{opt}) & \text{if } 0 < \mu < \textstyle\frac{1}{m} \\
(S,mS + k) : S^\text{min} \leq S \leq S^\text{max} & \text{if } \mu = \textstyle\frac{1}{m} \\
(S^{\text{min}},E^{\text{min}}) & \text{if } \mu > \textstyle\frac{1}{m} \end{cases},
\end{multline}
where $m = (E^\text{opt} - E^\text{min})/(S^\text{max} - S^\text{min})$ and $k = E^\text{min} - mS^\text{min}$.
\end{proposition}

\begin{proof}
Note that we can rewrite the maximization as a nested maximization problem
\begin{multline}
\max_{S_i,E_i} \{S_i - \mu E_i : (S_i,E_i) \in \mathcal{O}\} = \\ \max_{S_i}\{\max_{E_i}\{S_i - \mu E_i : (S_i,E_i) \in \mathcal{O} \} : (S_i, \cdot) \in \mathcal{O} \}.
\end{multline}
Now for a fixed value of $S_i$, the objective is strictly decreasing in $E_i$.  This means that the objective is maximized along the bottom boundary of the feasible set of operating points.  Specifically,
\begin{align}
E_i^*(S_i) & = \arg \max_{E_i}\{S_i - \mu E_i : (S_i,E_i) \in \mathcal{O} \} \\
& = mS_i + k.
\end{align}
Iterating on the nested maximization problem gives that the maximum is
\begin{equation}
S_i^* = \arg \max_{S_i}\{ S_i - \mu m S_i - \mu k : S_{\text{min}} \leq S_i \leq S_{\text{max}} \}.
\end{equation}
If $1 - \mu m > 0$ (or equivalently $\mu < 1/m$), then the objective is strictly increasing in $S_i$.  Alternatively, if $1 - \mu m = 0$ (or $\mu = 1/m$), then the objective is independent of $S_i$.  Lastly, if $1 - \mu m < 0$ (or also $\mu > 1/m$), then the objective is strictly decreasing in $S_i$.  The result follows by appropriate maximization of the corresponding objectives.
\end{proof}

\begin{remark}
This result is consistent with Theorem \ref{theorem:gen1}.
\end{remark}

\begin{remark}
The maximizers that correspond to $\mu = 0$ and $\mu = 1/m$ are not realistic reflections of the preference of the building owner.  This is because $\mu = 0$ corresponds to no interest in energy consumption, and $\mu = 1/m$ corresponds to a non-robust set of maximizers.  By non-robust, we mean that arbitrarily small perturbations of $\mu$ from $1/m$ lead to a qualitatively large change in the set of maximizers.
\end{remark}

The only maximizer of interest is $(S^{\text{max}},E^\text{opt})$; as a result, the only interesting range of values is $\mu : 0 < \mu < 1/m$.  This is because the point $(S^{\text{min}},E^{\text{min}})$ is equivalent to the maximizer of the optimization problem
\begin{equation}
\max_{S_i,E_i} \{ \textstyle\sum_i - E_i : (S_i,E_i) \in \mathcal{O} \},
\end{equation}
which only tries to minimize energy usage.  This is not reflective of the typical building owner, because the owners have external reasons for ensuring that occupants are comfortable.  For instance, satisfied occupants will be more productive with their work; furthermore, satisfied occupants will be more likely to remain within the building rather than relocating to another location.

\subsection{Moral Hazard with Baselining}
\label{sect:moral}

To remedy the discrepancy between the utility of owners and managers, there is interest in incentivizing building managers to consider energy savings.  One such proposal has the structure of a Stackelberg game (i.e., leader-follower pattern).  The owner first measures the energy consumption of the building during an initial ``baselining'' period $E_1$.  

After this initial period, the owner provides a monetary payment (or fine) with value given by the linear function $\gamma\cdot(E_1-E_2)$.  If less energy than the baseline energy is used $E_2 < E_1$, then a payment is given to the building manager; otherwise, when more energy than the baseline is used $E_2> E_1$, then the manager must pay a fine.  For this form of incentives, the corresponding utility is
\begin{multline}
\label{eqn:bad_inc}
\max_{S_i,E_i} \{ \gamma\cdot(E_1-E_2) + \textstyle\sum_i S_i - \lambda \cdot W(S_i,E_i) : \\
(S_i,E_i) \in \mathcal{O} \}.
\end{multline}

\begin{proposition}
\label{proposition:moral1}
Let $\lambda \geq 0$ be a fixed constant.  If $0 \leq \gamma^1 < \gamma^2$, then the corresponding maximizers of (\ref{eqn:bad_inc}) denoted $E^*_1$ are nondecreasing
\begin{equation}
E^{\text{max}} \geq \underline{E}^*_1(\gamma^2) \geq \overline{E}^*_1(\gamma^1),
\end{equation}
and the maximizers denoted $E^*_2$ are nonincreasing
\begin{equation}
\underline{E}^*_2(\gamma^1) \geq \overline{E}^*_2(\gamma^2) \geq E^\text{min}.
\end{equation}
\end{proposition}

\begin{proof}
This follows from Theorem \ref{theorem:gen2} and Corollary \ref{corollary:gen2}.
\end{proof}

\begin{proposition}
\label{proposition:moral2}
The maxmizers converge
\begin{align*}
\lim_{\gamma \rightarrow \infty} \underline{S}^*_1(\gamma) &= \lim_{\gamma \rightarrow \infty} \overline{S}^*_1(\gamma) = S^4 \\
\lim_{\gamma \rightarrow \infty} \underline{E}^*_1(\gamma) &= \lim_{\gamma \rightarrow \infty} \overline{E}^*_1(\gamma) = E^\text{max} \\
\lim_{\gamma \rightarrow \infty} \underline{S}^*_2(\gamma) &= \lim_{\gamma \rightarrow \infty} \overline{S}^*_2(\gamma) = S^\text{min} \\
\lim_{\gamma \rightarrow \infty} \underline{E}^*_2(\gamma) &= \lim_{\gamma \rightarrow \infty} \overline{E}^*_2(\gamma) = E^\text{min}.
\end{align*}
\end{proposition}

\begin{proof}
The proof is identical to that of Proposition \ref{proposition:converge}.
\end{proof}

\begin{remark}
The intuition of Propositions \ref{proposition:moral1} and \ref{proposition:moral2} is that a moral hazard occurs with the proposed incentive scheme in (\ref{eqn:bad_inc}).  During the first stage of baselining, the energy consumption will generally by greater than the true operating point of the building without incentives.  For high enough incentives, the energy consumption will be pushed towards the maximum possible $E^\text{max}$.  At the second stage, the energy consumption will generally be lower than the true operating point of the building without incentives; however, for very high levels of incentives, the building will be pushed towards a point in which satisfaction is not prioritized $(S^\text{min},E^\text{min})$.
\end{remark}

\section{Positive Incentive Design}

\begin{figure*}
\begin{subfigure}[b]{0.33\textwidth}
\centering
\includegraphics{main_points.pdf}
\caption{Key Points}
\label{fig:main_points4}
 \end{subfigure}
\begin{subfigure}[b]{0.33\textwidth}
\centering
\includegraphics{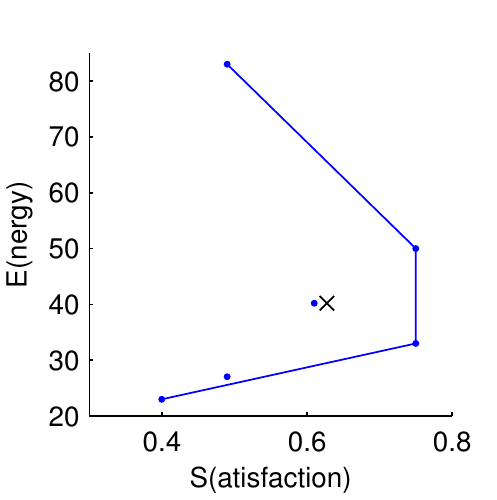}
\caption{Period 1 Trace of (\ref{eqn:our_inc})}
\label{fig:manager_trace4}
 \end{subfigure}
\begin{subfigure}[b]{0.33\textwidth}
\centering
\includegraphics{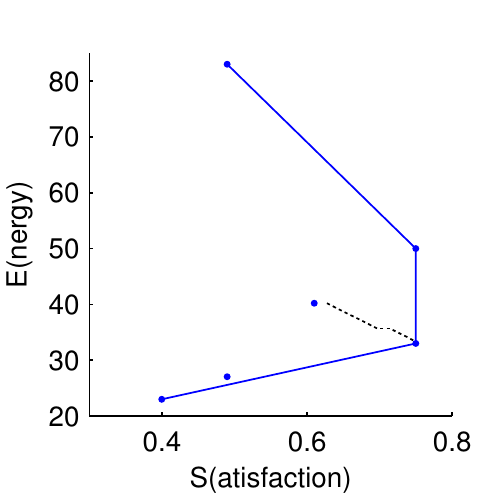}
\caption{Period 2 Trace of (\ref{eqn:our_inc})}
\label{figure:real4}
 \end{subfigure}

\caption{\label{fig:monte_model4}The key points of the static, operational model as determined by the Monte Carlo analysis are shown.  Furthermore, traces (indicated by a cross or the dotted line) of operating points as $\gamma$ is varied are shown for both period 1 and period 2 of operation of the incentive scheme in (\ref{eqn:our_inc}).}
\end{figure*}

As shown in Sect. \ref{sect:moral}, the proposed incentive scheme (\ref{eqn:bad_inc}) will encourage moral hazards; furthermore, it will likely encourage suboptimal operation in which satisfaction will also be minimized in order to reduce energy usage.  One interesting problem is to design an incentive that will encourage more efficient operation with high quality of service without the associated problems of moral hazards.

\subsection{Proposal and Properties}

We propose an incentive proportional to $S_2 - \kappa E_2$, with
\begin{equation}
\label{eqn:kappa}
\kappa = \min\{S^\text{min}/E^\text{max},(S^\text{max}-S^\text{min})/E^\text{max}\}.
\end{equation}
Note that $0 < \kappa < 1/m$, where $m$ is as defined in Sect. \ref{section:owner}.  With this incentive, the utility becomes
\begin{multline}
\label{eqn:our_inc}
\max_{S_i,E_i} \{ \gamma\cdot(S_2 - \kappa E_2) + \textstyle\sum_i S_i - \lambda \cdot W(S_i,E_i) : \\
(S_i,E_i) \in \mathcal{O} \},
\end{multline}
where $\gamma > 0$ is a constant.  This incentive is favorable.

\begin{proposition}
\label{proposition:positive}
If $(S_2,E_2) \in \mathcal{O}$, then $S_2 - \kappa E_2 \geq 0$.
\end{proposition}

\begin{proof}
Because $E_2 \leq E^\text{max}$, it must hold that $S_2 - \kappa E_2 \geq S_2 -  S^\text{min}$.  The result follows by noting that $S_2 \geq S^{\text{min}}$.
\end{proof}

\begin{proposition}
\label{proposition:our_converge}
The maximizers converge
\begin{align*}
\lim_{\gamma \rightarrow \infty} \underline{S}^*_2(\gamma) &= \lim_{\gamma \rightarrow \infty} \overline{S}^*_2(\gamma) = S^\text{max} \\
\lim_{\gamma \rightarrow \infty} \underline{E}^*_2(\gamma) &= \lim_{\gamma \rightarrow \infty} \overline{E}^*_2(\gamma) = E^\text{opt}.
\end{align*}
\end{proposition}

\begin{proof}
The proof is identical to that of Proposition \ref{proposition:converge}.
\end{proof}

\begin{remark}
The intuition of Proposition \ref{proposition:positive} is that the incentive is always positive; the building manager is never fined, regardless of the operation of the building.  As a result, the incentive can be thought of as a performance-based bonus.
\end{remark}

\begin{remark}
Proposition \ref{proposition:our_converge} says that as the incentive amount is increased, the building will eventually be configured at the optimal point $(S^\text{max},E^\text{opt})$.  This behavior is seen in Fig. \ref{fig:monte_model4}.
\end{remark}

There are two important implementation issues to mention.  First, it is not unreasonable to measure occupant satisfaction for the purpose of implementing the incentive: The proliferation of Internet and mobile technologies is reducing the costs and difficulties of conducting surveys to measure satisfaction.  Another advantage of this incentive is that it requires the use of conceptually simple parameters.  The size and usage of the building allows estimation of $E^\text{max}$ and $E^\text{opt}$; furthermore, an upper bound on $S^\text{max}$ and lower bound on $S^\text{min}$ are needed.

There is a last issue regarding the maximum payout.  In practice, it is not possible to provide an arbitrarily large incentive.  Let $P$ be the maximum payout that the building owner is willing to provide.  One possible value for the incentive is given by $\gamma = \mu P/(S^\text{max} - \kappa E^\text{opt})$, where $\mu$ is the elasticity of the building manager's utility with respect to a monetary payment.

\subsection{Simulation Analysis}


A pressing question is what monetary and energy savings are possible with the incentive scheme proposed in (\ref{eqn:our_inc}); a simulation is possible if $\mu$ is known.  Using dimensional analysis, we estimate $\mu = (S^*_1 - \lambda E^*_1)/R_1$, where $R_1$ is the monetary salary over period 1 and $(S^*_1,E^*_1)$ is the operating point of the building manager without any incentives.  For definiteness, we assume that each period spans one day.  

Under these assumptions, we can simulate the amount of savings per day as a function of the maximum payout $P$ and the price of energy: This is shown in Table \ref{table:simulation}.  The units of energy are in MWh/day, and the monetary values are in United States Dollars per day (USD/day).  For example, when energy costs \$100/MWh and maximum payout is \$200/day, then our analysis predicts 6.7MWh of energy savings per day and an increase in occupant comfort. This translates to a maximum payout over one year of \$73,000 and a savings of \$171,550 and 2445.5MWh over the span of one year.

\begin{table}[!h]
\centering
\begin{tabular}{|c||c|c|c|c|c|}
\hline
$\mathbf{P}$ & \textbf{\$0} & \textbf{\$50} & \textbf{\$100} & \textbf{\$150} & \textbf{\$200} \\
\hline
\hline
$\mathbf{E^*_2-E^*_1}$ & 0 & 0 & 0 & --4.5 & --6.7 \\
\hline
$\mathbf{S^*_2-S^*_1}$ & 0 & 0 & 0 & 0.1 & 0.1 \\
\hline
\textbf{Savings w/} & \multirow{2}{*}{\$0} & \multirow{2}{*}{--\$38} & \multirow{2}{*}{--\$75} & \multirow{2}{*}{--\$49} & \multirow{2}{*}{--\$64} \\
\textbf{\$20/MWh}&&&&&\\
\hline
\textbf{Savings w/} & \multirow{2}{*}{\$0} & \multirow{2}{*}{--\$38} & \multirow{2}{*}{--\$75} & \multirow{2}{*}{\$130} & \multirow{2}{*}{\$200} \\
\textbf{\$60/MWh}&&&&&\\
\hline
\textbf{Savings w/} & \multirow{2}{*}{\$0} & \multirow{2}{*}{--\$38} & \multirow{2}{*}{--\$75} & \multirow{2}{*}{\$310} & \multirow{2}{*}{\$470} \\
\textbf{\$100/MWh}&&&&&\\
\hline
\end{tabular}
\caption{Simulated Energy and Monetary Savings per Day with Incentive Scheme in (\ref{eqn:our_inc})}
\label{table:simulation}
\end{table}


These numerical results have several interesting characteristics.  First, there is an element of adverse selection that is present in this proposed incentive scheme.  Observe that when the maximum payout $P$ is \$50 or \$100, the change in energy usage and satisfaction is nearly zero; yet there is a bonus that is given.  Second, the monetary savings can be negative, which means that the bonus given to the building manager exceeds the money saved from reducing energy consumption for many combinations of payout values and energy costs.

These results have important implications for the implementation of incentive schemes such as (\ref{eqn:our_inc}): The proposed incentive scheme is not cost effective unless both the cost of energy is moderate and the maximum payout is high.  Because of the particularities of BRITE-S, the price of energy is currently about \$60/MWh.  This means that the incentive scheme is currently viable, but it requires a high payout.  The high payout required is due to a low elasticity of the building manager's utility with respect to a bonus --- this effectively reduces the value of the payout.  

At a broader level, our analysis suggests that all incentive schemes for reducing energy consumption will find difficulties in producing monetary savings without also providing high payouts.  The key problem is the high inelasticity of the manager's utility, and this will reduce the effectiveness of any incentive by essentially dampening its effect.

\section{Conclusion}

One proposed incentive scheme (\ref{eqn:bad_inc}) for encouraging efficient building operation was shown to suffer from moral hazard problems, in addition to encouraging reduced quality of service.  As a result, we propose an incentive amount
\begin{equation}
P\cdot\frac{S_2 - \kappa E_2}{S^\text{max} - \kappa E^\text{opt}},
\end{equation}
where $\kappa$ is as defined in (\ref{eqn:kappa}).This incentive scheme provides a tradeoff between quality of service and energy efficiency, and it likely has applicability to other systems (such as vehicle traffic networks \cite{aswani2011acc} where individual drivers take a role similar to that of a building manager) in which such a tradeoff is desired.  It has certain desirable properties, such as not encouraging moral hazards and always being nonegatively valued (so as to never penalize agents).





\section*{Acknowledgment}

The authors thank Galina Schwartz for discussions about electricity markets. 


\bibliographystyle{IEEEtran}
\bibliography{biblio}  

\end{document}